\documentclass[a4paper, 12pt]{amsart}
\usepackage[truedimen, right=22truemm, left=22truemm, top=18truemm,bottom=24truemm]{geometry}
\usepackage{amsmath}
\usepackage{amscd}
\usepackage{amssymb}
\usepackage{amsthm}
\usepackage{ascmac}
\usepackage[all]{xy}
\usepackage{color}

\theoremstyle{definition}
\newtheorem{thm}{Theorem}[section]

\newtheorem{lem}[thm]{Lemma}

\newtheorem{prop}[thm]{Proposition}
\newtheorem{conj}[thm]{Conjecture}

\newtheorem{rem}[thm]{Remark}

\def\Q{\mathbb{Q}}
\def\Z{\mathbb{Z}}
\def\Zp{\mathbb{Z}_{p}}
\def\Gal{\mathop{\mathrm{Gal}}\nolimits}
\def\jump#1{[\![{#1}]\!]}
\def\L{\Lambda}
\def\wt{\widetilde}
\def\p{\mathfrak{p}}
\def\mod{\mathop{\mathrm{mod}}\ }
\def\vphi{\varphi}
\def\ol{\overline}

\def\bp{\overline{\mathfrak{p}}}
\def\O{\mathcal{O}}
\def\P{\mathfrak{P}}
\def\bP{\overline{\mathfrak{P}}}
\def\rankZp{\mathop{\mathrm{rank}_{\Zp}}\nolimits}
\def\n{\nu}
\def\QQ{\mathfrak{Q}}
\def\q{\mathfrak{q}}

\numberwithin{equation}{section}

\begin{document}
\title{On Greenberg's generalized conjecture for imaginary quartic fields}
\author{Naoya Takahashi}
\address{Department of Mathematics, Faculty of Science, Kyoto University, Kyoto 606-8502, Japan}
\email{n.takahashi@math.kyoto-u.ac.jp}
\email{n.takahashi4017@gmail.com}
\date{January 31, 2020}

\maketitle

%\tableofcontents
%\clearpackage

\begin{abstract}
For an algebraic number field $K$ and a prime number $p$, let $\wt{K}/K$ be the maximal multiple $\Zp$-extension. Greenberg's generalized conjecture (GGC) predicts that the Galois group of the maximal unramified abelian pro-$p$ extension of $\wt{K}$ is pseudo-null over the completed group ring $\Zp\jump{\Gal(\wt{K}/K)}$. We show that GGC holds for some imaginary quartic fields containing imaginary quadratic fields and some prime numbers.
\end{abstract}

\section{Introduction}
Let $K$ be an algebraic number field and $p$ a prime number. Let $\wt{K}$ be the composite of all the $\Zp$-extensions of $K$. It is known that $\wt{K}/K$ is a $\Zp^{d}$-extension for some positive integer $d$. The completed group ring $\Zp\jump{\Gal(\wt{K}/K)}$ is isomorphic to the $d$-variable power series ring $\Zp\jump{T_{1},\dots,T_{d}}$. We denote by $L_{\wt{K}}$ the maximal unramified abelian pro-$p$ extension of $\wt{K}$. Then it is known that the Galois group $\Gal(L_{\wt{K}}/\wt{K})$ is a finitely generated torsion module over $\Zp\jump{\Gal(\wt{K}/K)}$; see \cite[Theorem 1]{Greenberg:1973}.

The following conjecture is often called Greenberg's Generalized Conjecture (GGC).

\begin{conj}[Greenberg {\cite[Conjecture (3.5)]{Greenberg:1998}}]
The Galois group $\Gal(L_{\wt{K}}/\wt{K})$ is a pseudo-null module over $\Zp\jump{\Gal(\wt{K}/K)}$ for every algebraic number field $K$ and every prime number $p$.
\end{conj}

Recall that a finitely generated $\Zp\jump{\Gal(\wt{K}/K)}$-module $M$ is called \textit{pseudo-null} if there are two non-zero annihilators $f,g \in \Zp\jump{\Gal(\wt{K}/K)}$ such that $f$ and $g$ are relatively prime. 

In this paper, we shall prove GGC for some imaginary quartic fields containing imaginary quadratic fields and some prime numbers.

\begin{thm}\label{main-theorem}
Let $F$ be an imaginary quadratic field and $p$ a prime number. Let $K/F$ be a quadratic extension of $F$. Assume that the following conditions are satisfied.

\begin{itemize}
\item The class number $h_{K}$ of $K$ is not divisible by $p$.
\item $p$ splits in $F/\Q$.
\item At least one of the primes of $F$ lying above $p$ does not split in $K/F$.
\end{itemize}
Then GGC holds for $K$ and $p$.
\end{thm}

\begin{rem}
GGC is known to be true for some number fields and some prime numbers. We briefly recall previous results. Let $K$ be a number field, $h_{K}$ the class number of $K$, and $K^{+}$ the maximal totally real subfield of $K$. We denote by $K_{\infty}^{+}$ the cyclotomic $\Zp$-extension of $K^{+}$.

\begin{enumerate}
\item (Minardi {\cite[Proposition 3.A]{Minardi:1986}}, see also {\cite[Theorem A, Theorem B]{Ozaki:1998}}) GGC holds when $K$ is an imaginary quadratic field and $p$ does not divide $h_{K}$.
\item (Bandini {\cite[Corollary 3.4, Theorem 3.6]{Bandini:2003}}) GGC holds when the following conditions are satisfied.
\begin{enumerate}
\item $p$ is an odd prime number.
\item $K$ is an imaginary biquadratic field containing two imaginary quadratic fields $E$ and $F$.
\item $p$ does not split in $E$ and does not divide the class number $h_{E}$ of $E$.
\item $p$ does not split in $K^{+}$ and GGC holds for $K^{+}$ and $p$.
\item GGC holds for $F$ and $p$.
\end{enumerate}
\item (Itoh {\cite[Theorem 1.1]{Itoh:2011}}, Fujii \cite[Theorem 2]{Fujii:2017}) GGC holds when the following conditions are satisfied.
\begin{enumerate}
\item $p$ is an odd prime number.
\item $K$ is a CM field such that $p$ splits completely in $K/\Q$.
\item $p$ does not divide the class number $h_{K}$ of $K$.
\item The Leopoldt conjecture holds for $K^{+}$ and $p$.
\item The Iwasawa invariants  $\lambda$, $\mu$ and $\nu$ of $K^{+}_{\infty}/K^{+}$ are zero.
\end{enumerate}
\item (Sharifi \cite[Theorem 1.3]{Sharifi:2008}) GGC holds when $K = \Q(\mu_{p})$ for $p < 1000$.
\item Kataoka proved GGC for certain complex cubic fields; see \cite{Kataoka:2017} for details.
\end{enumerate}
\end{rem}

\begin{rem}
Let $(F, K, p)$ be as in Theorem \ref{main-theorem}. Moreover, assume that $K$ is an imaginary biquadratic field and $p$ is an odd prime number. In this case, Theorem \ref{main-theorem} follows from Bandini's results in \cite{Bandini:2003}. In fact, the decomposition field of $p$ in $K/\Q$ is the imaginary quadratic field $F$. Let $E$ be the imaginary quadratic field different from $F$ contained in $K$. Then $p$ does not split in $E$ and $K^{+}$. Since $[K:E] = [K:F] = [K:K^{+}] = 2$ and $p$ is odd, $p$ does not divide the class numbers $h_{E}$, $h_{F}$, and $h_{K^{+}}$. Since $p$ does not split in $K^{+}/\Q$, the Iwasawa invariants  $\lambda$, $\mu$ and $\nu$ of $K^{+}_{\infty}/K^{+}$ are zero (see Lemma \ref{kK}). Hence GGC holds for $K^{+}$ and $p$. Finally, since $p$ does not divide $h_{F}$, GGC holds for $F$ and $p$ by Minardi's results {\cite[Proposition 3.A]{Minardi:1986}}.
\end{rem}

The outline of this paper is as follows. In Section 2, we give some preliminary results on units and Iwasawa modules. In Section 3, we shall explain the outline of the proof of Theorem \ref{main-theorem}. The proof is divided into three steps. In Section 4, 5, and 6, we give the proof of each step.

\section{Preliminaries}
Let $p$ be a prime number and $K$ an algebraic number field. We denote by $\O_{K}$ the ring of integers of $K$, by $E_{K} := \O_{K}^{\times}$ the group of units of $K$, and by $A_{K}$ the Sylow $p$-subgroup of the ideal class group of $K$. For a prime $\p$ of $K$, let $K_{\p}$ be the completion of $K$ at $\p$, and $\O_{K_{\p}}$ the ring of integers of $K_{\p}$. Let $U_{\p,1} := 1 +\p\O_{K_{\p}}$ be the group of principal units in $K_{\p}$. 

Let $S$ be a non-empty set of primes of $K$ lying above $p$. We put
$$E_{K}(S) := \{a \in E_{K} \mid a \equiv 1\  (\mod \p) \ \text{for all} \ \p \in S\}.$$
We define the diagonal map
$$\vphi_{S} \colon E_{K}(S) \to \prod_{\p \in S}U_{\p,1}, \ \ a \mapsto (a,\dots,a).$$
We denote by $\ol{\vphi_{S}(E_{K}(S))}$ the closure of the image $\vphi_{S}(E_{K}(S))$ in $\prod_{\p \in S}U_{\p,1}$. 
\par Let $M_{K}(S)$ be the maximal abelian pro-$p$ extension of $K$ unramified outside the primes in $S$. It is known that the Galois group $\Gal(M_{K}(S)/K)$ is a finitely generated $\Zp$-module. (See \cite[Chapter X]{Neukirch:2008}.)

\vspace{\baselineskip}
\begin{lem}\label{CFT}
There is an exact sequence
$$0 \to \frac{\prod_{\p \in S}U_{\p,1}}{\ \ol{\vphi_{S}(E_{K}(S))}\ } \to \Gal(M_{K}(S)/K) \to A_{K} \to 0$$
of finitely generated $\Zp$-modules.
\qed
\end{lem}

\begin{proof}
This follows from global class field theory. (For example, see \cite[Lemma 1]{Fujii:2017}.)
\end{proof}

For a finitely generated $\Zp$-module $M$, the \textit{$\Zp$-rank} of $M$ is defined to be the dimension of the vector space $M\otimes_{\Zp}\Q_{p}$ over $\Q_{p}$.
\begin{lem}\label{rank1}
Let $F$ be an imaginary quadratic field and $p$ a prime number which splits in $F/\Q$. Let $\p$ be a prime of $F$ lying above $p$. Let $K/F$ be a finite abelian extension. Let $S$ be the set of all the primes of $K$ lying above $\p$. Then the $\Zp$-rank of $\Gal(M_{K}(S)/K)$ is 1.
\end{lem}

\begin{proof}
By the same argument as in the proof of \cite[(10.3.6) Theorem]{Neukirch:2008} and \cite[(10.3.16) Theorem]{Neukirch:2008}, the $\Zp$-rank of 
$$\frac{\prod_{\P \in S}U_{\P,1}}{\ \ol{\varphi_{S}(E_{K}(S))}\ }$$
 is 1. Hence the assertion follows by Lemma \ref{CFT}.
\end{proof}

\begin{lem}[Iwasawa \cite{Iwasawa:1956}]\label{kK}
Let $K$ be an algebraic number field and $p$ a prime number. Let $L/K$ be a cyclic extension of $p$-power degree. Assume that there is a prime $\p$ of $K$ lying above $p$ such that $L/K$ is unramified outside $\p$. If the class number $h_{L}$ of $L$ is divisible by $p$, then the class number $h_{K}$ of $K$ is also divisible by $p$.
\end{lem}

\begin{proof}
If $\p$ is not totally ramified in $L/K$, let $I_{\p} \subset \Gal(L/K)$ be the inertia group at $\p$. Then $L^{I_{\p}}/K$ is unramified everywhere. By global class field theory, the class number $h_{K}$ is divisible by $p$. If $\p$ is totally ramified in $L/K$, see \cite[II]{Iwasawa:1956}.
\end{proof}

\begin{lem}[Perrin-Riou \cite{Perrin-Riou:1984}]\label{G/H->G}
Let $M$ be a finitely generated torsion module over $\Zp\jump{T_{1}, \dots, T_{d}}$ for some $d \geq 2$. If $M/T_{d}M$ is a pseudo-null module over $\Zp\jump{T_{1}, \dots, T_{d-1}}$, then $M$ is a pseudo-null module over $\Zp\jump{T_{1}, \dots, T_{d}}$. 
\end{lem}

\begin{proof}
See \cite[Proposition 3.1]{Bandini:2003} and \cite[Chapter 4.B]{Minardi:1986}. See also \cite[Lemma 2.2]{Kataoka:2017}. (For the original proof, see \cite[Lemme 2]{Perrin-Riou:1984}.)
\end{proof}

\section{Outline of the proof of Theorem \ref{main-theorem}.}
Let $F$ be an imaginary quadratic field. Let $p$ be a prime number which splits in $F/\Q$. Let $(p) = \p_{F}\bp_{F}$ be the decomposition of $p$ in $F$. Assume that $\p_{F}$ does not split in $K/F$. Let $\p_{K}$ be the unique prime of $K$ lying above $\p_{F}$. Note that $\bp_{F}$ may or may not split in $K/F$. 

There is a unique $\Zp$-extension $F^{(1)}$ of $F$ such that $F^{(1)}/F$ is unramified outside $\p_{F}$ by Lemma \ref{rank1}. Similarly, there is a unique $\Zp$-extension $F'^{(1)}$ of $F$ such that $F'^{(1)}/F$ is unramified outside $\bp_{F}$.

We put 
$$F^{(2)}:= F^{(1)}F'^{(1)}, \quad K^{(1)} := KF^{(1)}, \quad K^{(2)} := KF^{(2)}.$$

The following conditions are satisfied.
\begin{enumerate}
\item $K^{(1)}/K$ is a $\Zp$-extension unramified outside $\p_{K}$.
\item $K^{(2)}/K$ is a $\Zp^{2}$-extension unramified outside $\p_{K}$ and the primes of $K$ lying above $\bp_{F}$.
\item $K^{(2)}$ contains $K^{(1)}$, and $K^{(2)}/K^{(1)}$ is a $\Zp$-extension unramified outside the primes of $K^{(1)}$ lying above $\bp_{F}$.
 \end{enumerate}

%Since $F$ is an imaginary quadratic field, there is a unique $\Zp^{2}$-extension $F^{(2)}$ of $F$. We put $K^{(2)} := KF^{(2)}$. Then it follows that $K^{(1)} \subset K^{(2)}$ since $K^{(1)} = KF^{(1)}$. Note that $K^{(2)}/K^{(1)}$ is unramified outside the primes of $K^{(1)}$ lying above $\bp_{K}$ since $F^{(2)}/F^{(1)}$ is unramified outside the primes of $F^{(1)}$ lying above $\bp_{F}$.

Since $K$ is a quadratic extension of $F$, the Leopoldt conjecture holds for $K$ and $p$ ; see \cite[(10.3.16) Theorem]{Neukirch:2008}. Therefore, there is a unique $\Zp^{3}$-extension $K^{(3)}$ of $K$.

 For $i = 1, 2, 3$, let $L^{(i)}$ be the maximal unramified abelian pro-$p$ extension of $K^{(i)}$. We put 
 $$X^{(i)} := \Gal(L^{(i)}/K^{(i)}), \quad \L^{(i)} :=  \Zp\jump{\Gal(K^{(i)}/K)}.$$
 Note that $\L^{(i)}$ is isomorphic to the $i$-variable power series ring $\Zp\jump{T_{1},\dots,T_{i}}$. 

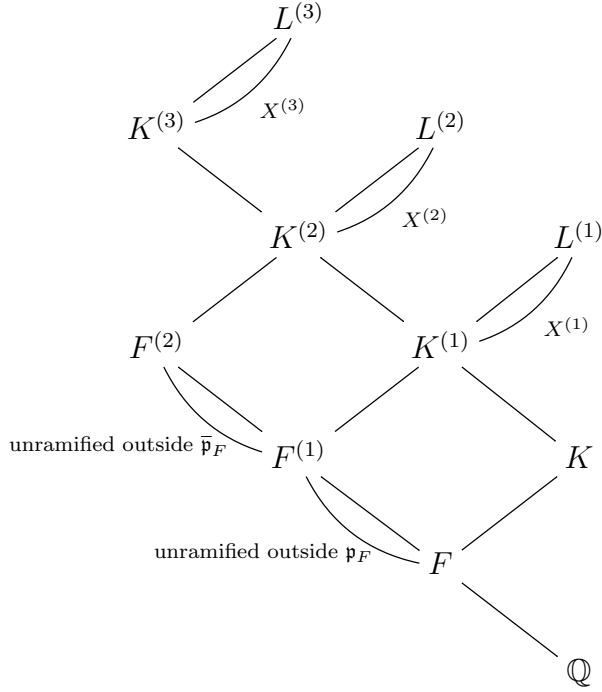
\[
	\xymatrix{
		&L^{(3)} \ar@{-}[dl] \ar@{-}@/^{10pt}/[dl]^--{X^{(3)}}&&\\
		K^{(3)} \ar@{-}[dr]  &&L^{(2)} \ar@{-}[dl] \ar@{-}@/^{10pt}/[dl]^--{X^{(2)}}&\\
		&K^{(2)}\ar@{-}[dl] \ar@{-}[dr] &&L^{(1)} \ar@{-}[dl] \ar@{-}@/^{10pt}/[dl]^--{X^{(1)}}\\
		F^{(2)} \ar@{-}[dr] \ar@{-}@/_{10pt}/[dr]_{\text{unramified outside $\bp_{F}$}}&&K^{(1)} \ar@{-}[dl] \ar@{-}[dr]&\\
		&F^{(1)} \ar@{-}[dr] \ar@{-}@/_{10pt}/[dr]_{\text{unramified outside $\p_{F}$}}&&K \ar@{-}[dl]\\
		&&F \ar@{-}[dr]&\\
		&&&\Q
	}
\]

The outline of the proof of Theorem \ref{main-theorem} is as follows.

\begin{description}
\item[Step 1] First, we shall show that $X^{(1)} = 0$.

\item[Step 2] Then, we shall show that $X^{(2)}$ is a pseudo-null module over $\L^{(2)}$.

\item[Step 3] Finally, we shall show that $X^{(3)}$ is a pseudo-null module over $\L^{(3)}$.
\end{description}

\section{Step 1: The $\Zp$-extension $K^{(1)}/K$}
\begin{prop}\label{X=0}
$X^{(1)} = 0.$
\end{prop}

\begin{proof}
Since $K^{(1)}/K$ is unramified outside $\p_{K}$ and $p$ does not divide the class number $h_{K}$ of $K$, it follows from Lemma \ref{kK} that $X^{(1)} = 0$.
\end{proof}

\section{Step 2: The $\Zp^{2}$-extension $K^{(2)}/K$}

Let $L_{2}$ be the maximal abelian extension of $K^{(1)}$ contained in $L^{(2)}$. Then $\Gal(L_{2}/K^{(2)})$ is the maximal quotient of $X^{(2)}$ on which $\Gal(K^{(2)}/K^{(1)})$ acts trivially. 

Let $S^{(1)}$ be the set of all the primes of $K^{(1)}$ lying above $\bp_{F}$.

For each $\bP \in S^{(1)}$, let $I_{\bP} \subset \Gal(L_{2}/K^{(1)})$ be the inertia subgroup at $\bP$. Then we have
$$\sum_{\bP \in S^{(1)}}I_{\bP} = \Gal(L_{2}/K^{(1)})$$
since $X^{(1)} = 0$ (see Proposition \ref{X=0}) and $L_{2}/K^{(1)}$ is unramified outside $S^{(1)}$.

\begin{lem}\label{finitely-decomposed}
\begin{enumerate}
\item $\p_{K}$ is totally ramified in $K^{(1)}/K$, and finitely decomposed in $KF'^{(1)}/K$.
\item Every prime of $K$ lying above $\bp_{F}$ is finitely decomposed in $K^{(1)}/K$.
\item Every prime of $K$ lying above $p$ is finitely decomposed in $K^{(2)}/K$.
\end{enumerate}
\end{lem}

\begin{proof}
First, we shall show (1). Since $K^{(1)}/K$ is unramified outside $\p_{K}$ and $p$ does not divide the class number $h_{K}$ of $K$, the prime $\p_{K}$ is totally ramified in $K^{(1)}/K$.  By \cite[Lemma 3]{Fujii:2017}, $\p_{F}$ is finitely decomposed in $F'^{(1)}/F$. Hence $\p_{K}$ is finitely decomposed in $KF'^{(1)}/K$. 

Next, we shall show (2). By \cite[Lemma 3]{Fujii:2017}, $\bp_{F}$ is finitely decomposed in $F^{(1)}/F$. Hence every prime of $K$ lying above $\bp_{F}$ is finitely decomposed in $K^{(1)}/K$. 

Finally, the assertion (3) follows from (1) and (2) since the ramification index of $\bp_{F}$ in $F'^{(1)}/F$ is infinite.
\end{proof}

\begin{lem}
There is an intermediate field $K_{n}$ of $K^{(1)}/K$ such that for every $\bP \in S^{(1)}$, the inertia group $I_{\bP}$ is a $\Zp\jump{\Gal(K^{(1)}/K_{n})}$-submodule of $\Gal(L_{2}/K^{(1)})$.\end{lem}

\begin{proof}
For a prime $\q$ of $K$ lying above $\bp_{F}$, let $D_{\q} \subset \Gal(K^{(1)}/K)$ be the decomposition group at $\q$. By Lemma \ref{finitely-decomposed} (2), the index of $D_{\q}$ in $\Gal(K^{(1)}/K)$ is finite. We take a sufficiently large $n$ such that
$$p^{n} \geq [\Gal(K^{(1)}/K) : D_{\q}]$$
for every prime $\q$ of $K$ lying above $\bp_{F}$. Let $K_{n}$ be the intermediate field of $K^{(1)}/K$ such that $[K_{n}:K] = p^{n}$. Since $K^{(1)}/K$ is unramified over every prime of $K$ lying above $\bp_{F}$, every prime of $K_{n}$ lying above $\bp_{F}$ remains prime in $K^{(1)}$. This shows that $I_{\bP}$ is a $\Zp\jump{\Gal(K^{(1)}/K_{n})}$-submodule of $\Gal(L_{2}/K^{(1)})$ for every prime $\bP \in S^{(1)}$. 
\end{proof}

\begin{lem}
$L_{2}$ is an abelian extension of $K_{n}$.
\end{lem}

\begin{proof}
For each $\bP \in S^{(1)}$, the restriction map 
$$I_{\bP} \to \Gal(K^{(2)}/K^{(1)})$$
 is injective since $L_{2}/K^{(2)}$ is an unramified extension.  The action of $\Gal(K^{(1)}/K_{n})$ on $\Gal(K^{(2)}/K^{(1)})$ is trivial since the extension $K^{(2)}/K_{n}$ is abelian. Therefore, we see that $\Gal(K^{(1)}/K_{n})$ acts trivially on $I_{\bP}$ for each prime $\bP \in S^{(1)}$. Hence $\Gal(K^{(1)}/K_{n})$ acts trivially on 
 $$\sum_{\bP \in S^{(1)}}I_{\bP} = \Gal(L_{2}/K^{(1)}). $$
 Consequently, $L_{2}/K_{n}$ is an abelian extension. 
\end{proof}

Note that $\Gal(L_{2}/K_{n})$ is a finitely generated $\Zp$-module since $K_{n}$ is a finite extension of $\Q$. (See \cite[Chapter X]{Neukirch:2008}.)

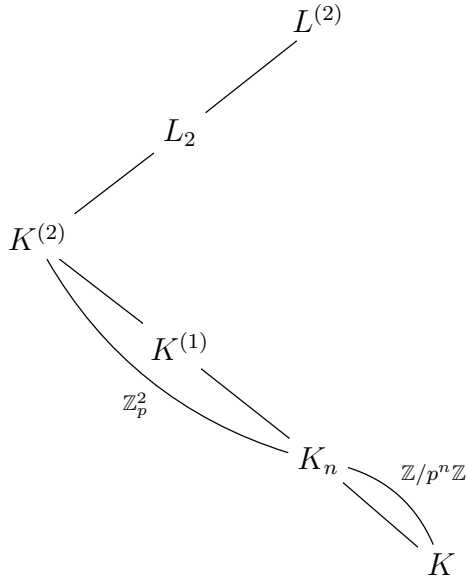
\[
	\xymatrix{
		&&L^{(2)} \ar@{-}[dl]\\
		&L_{2} \ar@{-}[dl] &\\
		K^{(2)} \ar@{-}[dr] \ar@{-}@/_{15pt}/[ddrr]_{\Zp^{2}}&&\\
		&K^{(1)} \ar@{-}[dr]&\\
		&&K_{n} \ar@{-}[dr] \ar@{-}@/^{10pt}/[dr]^{\Z/p^{n}\Z}&\\
		&&&K
	}
\]

\vspace{\baselineskip}
\begin{lem}\label{rank2}
The $\Zp$-rank of $\Gal(L_{2}/K_{n})$ is 2.
\end{lem}

\begin{proof}
Since $K_{n} \subset K^{(2)} \subset L_{2}$ and $K^{(2)}/K_{n}$ is a $\Zp^{2}$-extension, it is clear that 
$$\rankZp \Gal(L_{2}/K_{n}) \geq 2.$$

It is enough to show the opposite inequality. By Lemma \ref{finitely-decomposed} (1), $\p_{K}$ is totally ramified in $K^{(1)}/K$. Let $I_{n} \subset \Gal(L_{2}/K_{n})$ be the inertia group of $L_{2}/K_{n}$ for the unique prime of $K_{n}$ lying above $\p_{K}$. Since $L_{2}/K^{(1)}$ is unramified at the unique prime of $K^{(1)}$ lying above $\p_{K}$, the restriction map
$$I_{n} \to \Gal(K^{(1)}/K_{n}) \cong \Zp$$
is injective. Hence the $\Zp$-rank of $I_{n}$ is 1. Let $M_{2} := (L_{2})^{I_{n}}$ be the fixed field of $L_{2}$ by $I_{n}$. Then we have the following exact sequence:
$$0 \to I_{n} \to \Gal(L_{2}/K_{n}) \to \Gal(M_{2}/K_{n}) \to 0.$$
Here $M_{2}$ is an abelian pro-$p$ extension of $K_{n}$ which is unramified outside the primes lying above $\bp_{F}$. Since $K_{n}$ is a finite abelian extension of the imaginary quadratic field $F$, the $\Zp$-rank of $\Gal(M_{2}/K_{n})$ is at most 1 by Lemma \ref{rank1}. Therefore, we have
\begin{align*}
\rankZp \Gal(L_{2}/K_{n}) &= \rankZp I_{n} + \rankZp \Gal(M_{2}/K_{n})\\
&\leq 1 + 1\\
&= 2.
\end{align*}
\end{proof}

\begin{prop}\label{step2}
$X^{(2)}$ is a pseudo-null module over $\L^{(2)}$.
\end{prop}

\begin{proof}
We have
$$\rankZp \Gal(L_{2}/K_{n}) = \rankZp \Gal(L_{2}/K^{(2)}) + \rankZp \Gal(K^{(2)}/K_{n}).$$
By Lemma \ref{rank2}, we have 
$$\rankZp \Gal(L_{2}/K_{n}) = 2.$$
Since $K^{(2)}/K_{n}$ is a $\Zp^{2}$-extension, we have 
$$\rankZp \Gal(K^{(2)}/K_{n}) = 2.$$
 It follows that $\Gal(L_{2}/K^{(2)})$ is a finite abelian group. Hence $\Gal(L_{2}/K^{(2)})$ is a pseudo-null module over $\L^{(1)}$. Therefore, $X^{(2)}$ is a pseudo-null module over $\L^{(2)}$ by Lemma \ref{G/H->G}.
\end{proof}

\section{Step 3: Proof of Theorem \ref{main-theorem}}
Let $L_{3}$ be the maximal abelian extension of $K^{(2)}$ contained in $L^{(3)}$. Similarly as in Step 2, $\Gal(L_{3}/K^{(3)})$ is the maximal quotient of $X^{(3)} = \Gal(L^{(3)}/K^{(3)})$ on which $\Gal(K^{(3)}/K^{(2)})$ acts trivially.

\[
	\xymatrix{
		&&L^{(3)} \ar@{-}[dl] \ar@{-}@/_{12pt}/[ddll]_{X^{(3)}}\\
		&L_{3} \ar@{-}[dl] \ar@{-}[dr]&\\
		K^{(3)} \ar@{-}[dr]&&L^{(2)} \ar@{-}[dl] \ar@{-}@/^{12pt}/[dl]^{X^{(2)}}&\\
	 	&K^{(2)}
	}
\]

\begin{lem}\label{pseudo-null/2}
$\Gal(L_{3}/K^{(2)})$ is a pseudo-null module over $\L^{(2)}$.
\end{lem}

\begin{proof}
We follow the same argument as in the proof of \cite[Theorem 3.6]{Bandini:2003}. We have an exact sequence
$$0 \to \Gal(L_{3}/L^{(2)}) \to \Gal(L_{3}/K^{(2)}) \to X^{(2)} = \Gal(L^{(2)}/K^{(2)}) \to 0.$$
By Proposition \ref{step2}, $X^{(2)}$ is a pseudo-null module over $\L^{(2)}$. Hence it is enough to show that $\Gal(L_{3}/L^{(2)})$ is a pseudo-null module over $\L^{(2)}$. 

For every prime $\QQ$ of $K^{(2)}$ lying above $p$, let $I_{\QQ} \subset \Gal(L_{3}/K^{(2)})$ be the inertia group at $\QQ$. Since $L^{(2)}$ is the maximal unramified extension of $K^{(2)}$ contained in $L_{3}$, we have
$$\Gal(L_{3}/L^{(2)}) = \sum_{\QQ \mid p}I_{\QQ}.$$
Since $L_{3}/K^{(3)}$ is unramified everywhere, the restriction map 
$$I_{\QQ} \to \Gal(K^{(3)}/K^{(2)}) \cong \Zp$$
is injective. Therefore, we see that $I_{\QQ}$ is 0 or isomorphic to $\Zp$.

Recall that every prime of $K$ lying above $p$ is finitely decomposed in $K^{(2)}$; see Lemma \ref{finitely-decomposed} (3). Let $\q$ be a prime of $K$ lying above $p$. The $\Zp$-rank of the decomposition group $D_{\q} \subset \Gal(K^{(2)}/K) \cong\Zp^{2}$ is 2. Hence $D_{\q}$ is isomorphic to $\Zp^{2}$. Let $\n_{1}(\q)$, $\n_{2}(\q)$ be two independent topological generators of $D_{\q}$. Let $\QQ$ be a prime of $K^{(2)}$ lying above $\q$. Since $D_{\q}$ fixes $\QQ$, it acts on $I_{\QQ}$. The action of $D_{\q}$ on $I_{\QQ}$ is trivial because $\Gal(K^{(3)}/K) \cong \Zp^{3}$ is abelian and $I_{\QQ}$ is mapped injectively into $\Gal(K^{(3)}/K^{(2)})$. Therefore $\n_{1}(\q)-1$ and $\n_{2}(\q)-1$ correspond to two relatively prime elements of 
$$\L^{(2)} = \Zp\jump{\Gal(K^{(2)}/K)} \cong \Zp\jump{T_{1},T_{2}}$$
which annihilate $\sum_{\QQ \mid \q}I_{\QQ}$. Hence  $\sum_{\QQ \mid \q}I_{\QQ}$ is a pseudo-null $\L^{(2)}$-module. 

Since
$$\Gal(L_{3}/L^{(2)}) = \sum_{\QQ \mid p}I_{\QQ} = \sum_{\text{$\q$ : prime of $K$ lying above $p$}}\left(\sum_{\QQ \mid \q}I_{\QQ}\right),$$
we conclude that $\Gal(L_{3}/L^{(2)})$ is a pseudo-null module over $\L^{(2)}$.
\end{proof}

\begin{proof}(Proof of Theorem \ref{main-theorem}.)
We see that $X^{(3)}$ is pseudo-null module over $\L^{(3)}$ by Lemma \ref{pseudo-null/2} and Lemma \ref{G/H->G}. The proof of Theorem \ref{main-theorem} is complete.
\end{proof}

\subsection*{Acknowledgements}
The author would like to express my appreciation to Tetsushi Ito for constructive suggestions and encouragement. The author is also particularly grateful to Satoshi Fujii and Tatsuya Ohshita for helpful discussions and comments on an earlier version of this paper.


\begin{thebibliography}{Z}
\bibitem{Bandini:2003} Bandini, A., \textit{Greenberg's conjecture for $\Zp^{d}$-extensions,} Acta Arith. 108 (2003), no. 4, 357-368. 
\bibitem{Fujii:2017}Fujii, S., \textit{On Greenberg's generalized conjecture for CM-fields, }J. Reine Angew. Math. 731 (2017), 259-278.
\bibitem{Greenberg:1973} Greenberg, R., \textit{The Iwasawa invariants of Γ-extensions of a fixed number field,} Amer. J. Math. 95 (1973), 204-214. 
\bibitem{Greenberg:1998}Greenberg, R., \textit{Iwasawa theory-past and present,} Class field theory-its centenary and prospect (Tokyo, 1998), 335-385, Adv. Stud. Pure Math., 30, Math. Soc. Japan, Tokyo, 2001.
\bibitem  {Itoh:2011}Itoh, T., \textit{On multiple $\Zp$-extensions of imaginary abelian quartic fields,} J. Number Theory 131 (2011), no. 1, 59-66. 
\bibitem{Iwasawa:1956}Iwasawa, K., \textit{A note on class numbers of algebraic number fields}, Abh. Math. Sem. Univ. Hamburg 20 (1956), 257-258. 
\bibitem{Kataoka:2017}Kataoka, T., \textit{On Greenberg's generalized conjecture for complex cubic fields,} Int. J. Number Theory 13 (2017), no. 3, 619-631.
\bibitem{Minardi:1986}Minardi, J. V., \textit{Iwasawa modules for $\Zp^{d}$-extensions of algebraic number fields,} Thesis (Ph.D.)-University of Washington. 1986. 77 pp.
\bibitem{Neukirch:2008}Neukirch, J., Schmidt, A., Wingberg, K., \textit{Cohomology of number fields, Second edition.} Grundlehren der Mathematischen Wissenschaften, 323. Springer-Verlag, Berlin, 2008. 
\bibitem{Ozaki:1998}Ozaki, M., \textit{Iwasawa invariants of $\Zp$-extensions over an imaginary quadratic field,} Class field theory-its centenary and prospect (Tokyo, 1998), 387-399, Adv. Stud. Pure Math., 30, Math. Soc. Japan, Tokyo, 2001. 
\bibitem{Perrin-Riou:1984}Perrin-Riou, B., \textit{Arithm\'{e}ctique des courbes elliptiques et th\'{e}orie d'Iwasawa, } M\'{e}m. Soc. Math. France (N.S.) No. 17 (1984), 130 pp.
\bibitem{Sharifi:2008} Sharifi, R. T., \textit{On Galois groups of unramified pro-p extensions,} Math. Ann. 342 (2008), no. 2, 297-308. 
\end{thebibliography}
\end{document}